\documentclass[a4paper,reqno,twoside]{amsart} %nosumlimits,nonamelimits,
\usepackage{amssymb}
\usepackage{latexsym}
\usepackage{amsmath}
\usepackage{mathrsfs}
\usepackage[all,2cell]{xy}
        \UseAllTwocells \SilentMatrices
\usepackage{ifthen}

\usepackage{comment}
\theoremstyle{plain}
\newtheorem{propn}{Proposition}[section]
\newtheorem{thm}[propn]{Theorem}

\newtheorem{cor}[propn]{Corollary}

\theoremstyle{definition}
\newtheorem{remarks}[propn]{Remarks}
\newtheorem{defn}[propn]{Definition}
\newtheorem*{notn}{Notation}
\newtheorem*{term}{Terminology}

\newtheorem{example}[propn]{Example}

\theoremstyle{remark}

\newcommand{\ve}{\varepsilon}

\newcommand{\Lindblad}{\mathcal{L}}
\newcommand{\LV}{L^V}
\newcommand{\LtildeV}{\Ltilde^V}
\newcommand{\CV}{C^V}
\newcommand{\Vdual}{\wt{V}}
\newcommand{\Ltilde}{\wt{L}}
\newcommand{\calQ}{\mathcal{Q}}
\newcommand{\gV}{\gamma^V}
\newcommand{\Xh}{\mathfrak{X}(\init)}
\newcommand{\Xholh}{\mathfrak{X}_{\mathrm{hol}}(\init)}
\newcommand{\Xholhotkhat}{\mathfrak{X}_{\mathrm{hol}}(\init\ot\khat)}
\newcommand{\Xtwohk}{\mathfrak{X}^{(2)}(\init,\noise)}
\newcommand{\Xtwoholhk}{\mathfrak{X}^{(2)}_{\mathrm{hol}}(\init,\noise)}

\newcommand{\Xfourholhk}{\mathfrak{X}^{(4)}_{\mathrm{hol}}(\init,\noise)}
\newcommand{\Xfourholhki}{\mathfrak{X}^{(4)}_{\mathrm{hol}}(\init,\noise_i)}
\newcommand{\QSCVET}{\mathbb{QSC}(\Vil\!:\Exps_\Til)}
\newcommand{\QSCddaggerVETprimeET}{\mathbb{QSC}^{\ddagger}(\Vil\!:\Exps_{\Til'},\Exps_\Til)}
\newcommand{\QSCMrVET}{\mathbb{QSC}_{\mathrm{Mreg}}(\Vil\!:\Exps_\Til)}
\newcommand{\QSCMrAE}{\mathbb{QSC}_{\mathrm{Mreg}}(\Al\!:\Exps)}
\newcommand{\QSCcbMrVET}{\mathbb{QSC}_{\mathrm{cbMreg}}(\Vil\!:\Exps_\Til)}
\newcommand{\QSCVdaggerETprime}{\mathbb{QSC}(\Vil^\dagger\!:\Exps_{\Til'})}
\newcommand{\QSChbHcVET}{\mathbb{QSC}_{\hil\text{-}\mathrm{bHc}}(\Vil\!:\Exps_\Til)}

\newcommand{\QSCddaggercbHcAEE}{\mathbb{QSC}^\ddagger_{\mathrm{cbHc}}(\Al\!:\Exps,\Exps)}

\newcommand{\QSChbHcVdaggerETprime}{\mathbb{QSC}_{\hil\text{-}\mathrm{bHc}}(\Vil^\dagger\!:\Exps_{\Til'})}
\newcommand{\QSCkhatbHcVET}{\mathbb{QSC}_{\khat\text{-}\mathrm{bHc}}(\Vil\!:\Exps_\Til)}

\newcommand{\QSCcbHcVET}{\mathbb{QSC}_{\mathrm{cbHc}}(\Vil\!:\Exps_\Til)}
\newcommand{\QSCddaggerhbHcVETprimeET}{\mathbb{QSC}^\ddagger_{\hil\text{-}\mathrm{bHc}}(\Vil\!:\Exps_{\Til'},\Exps_\Til)}
\newcommand{\QSCddaggerkhatbHcVETprimeET}{\mathbb{QSC}^\ddagger_{\khat\text{-}\mathrm{bHc}}(\Vil\!:\Exps_{\Til'},\Exps_\Til)}
\newcommand{\QSCddaggercbHcVETprimeET}{\mathbb{QSC}^\ddagger_{\mathrm{cbHc}}(\Vil\!:\Exps_{\Til'},\Exps_\Til)}

\newcommand{\Pcd}{P^{c,d}}
\newcommand{\calPcd}{\mathcal{P}^{c,d}}
\newcommand{\calPcc}{\mathcal{P}^{c,c}}
\newcommand{\QDS}{\mathcal{T}}
\newcommand{\minimalKL}{\mathcal{T}^{(K,L)}}

\newcommand{\QScChk}{\mathbb{QS}_{\mathrm{c}}\mathbb{C}(\init,\noise)}
\newcommand{\QScCMrhk}{\mathbb{QS}_{\mathrm{c}}\mathbb{C}_{\mathrm{Mreg}}(\init,\noise)}
\newcommand{\QSuCMrhk}{\mathbb{QS}_{\mathrm{u}}\mathbb{C}_{\mathrm{Mreg}}(\init,\noise)}
\newcommand{\QScCholhk}{\mathbb{QS}_{\mathrm{c}}\mathbb{C}_{\mathrm{hol}}(\init,\noise)}
\newcommand{\noise}{\mathsf{k}}

\newcommand{\vp}{\varpi}

\newcommand{\Nil}{\mathsf{N}}

\newcommand{\Al}{\mathsf{A}}

\newcommand{\Hil}{\mathsf{H}}
\newcommand{\hil}{\mathsf{h}}
\newcommand{\Kil}{\mathsf{K}}
\newcommand{\kil}{\mathsf{k}}
\newcommand{\Vil}{\mathsf{V}}

\newcommand{\Wil}{\mathsf{W}}

\newcommand{\init}{\mathfrak{h}}
\newcommand{\Til}{\mathsf{T}}
\newcommand{\Dil}{\mathsf{D}}

\newcommand{\Exps}{\mathcal{E}}
\newcommand{\Fock}{\mathcal{F}}

\newcommand{\Cstar}{C^*}

\newcommand{\khat}{\wh{\kil}}

\newcommand{\chat}{\wh{c}}
\newcommand{\dhat}{\wh{d}}

\newcommand{\Step}{\mathbb{S}}
\newcommand{\Real}{\mathbb{R}}
\newcommand{\Rplus}{\Real_+}
\newcommand{\Comp}{\mathbb{C}}
\newcommand{\Nat}{\mathbb{N}}

\newcommand{\cb}{{\text{\tu{cb}}}}

\newcommand{\wh}{\widehat}
\newcommand{\wt}{\widetilde}
\newcommand{\ol}{\overline}
\newcommand{\ul}{\underline}
\newcommand{\ot}{\otimes}
\newcommand{\otM}{\otimes_{\mathrm{M}}}

\newcommand{\otol}{\,\ol{\otimes}\,}
\newcommand{\otul}{\,\ul{\otimes}\,}

\newcommand{\op}{\oplus}
\newcommand{\la}{\langle}
\newcommand{\ra}{\rangle}

\newcommand{\tu}{\textup}

\DeclareMathOperator{\Dom}{Dom}
\DeclareMathOperator{\Ran}{Ran}
\DeclareMathOperator{\Lin}{Lin}

\DeclareMathOperator{\id}{id}
\DeclareMathOperator{\re}{Re}
\DeclareMathOperator{\im}{Im}

\newcommand{\ip}[2]{\langle #1, #2 \rangle}
\newcommand{\bip}[3][\big]{#1\langle #2, #3 #1\rangle}
\newcommand{\norm}[1]{\lVert #1 \rVert}
\newcommand{\cbnorm}[1]{\lVert #1 \rVert_\cb}

\newcommand{\bra}[1]{\langle #1 |}
\newcommand{\ket}[1]{| #1 \rangle}

\newenvironment{alist}
{

\begin{enumerate}}
{\end{enumerate}}

\newenvironment{rlist}
{

\begin{enumerate}}
{\end{enumerate}}

\numberwithin{equation}{section}
\begin{document}

\title
[Differentiating quantum stochastic cocycles]
 {How to differentiate 
\\ a quantum stochastic cocycle}

\author[Martin Lindsay]{J.\ Martin Lindsay}
\address{Department of Mathematics and Statistics,
Lancaster University, Lancaster LA1 4YF, U.K.}
\email{j.m.lindsay@lancaster.ac.uk}

\dedicatory{Dedicated to Robin Hudson on his seventieth birthday}

\subjclass[2000]{Primary 46L53, 81S25; Secondary 47D06}
 
 \keywords{Noncommutative probability, quantum stochastic cocycle, $E_0$-semigroup,
 CCR flow, holomorphic semigroup}

 \begin{abstract}
 Two new approaches to the infinitesimal characterisation of quantum
 stochastic cocycles are reviewed.
 The first concerns mapping cocycles on an operator space and
 demonstrates the role of H\"older continuity;
 the second concerns contraction operator cocycles on a Hilbert space and shows
 how holomorphic assumptions yield cocycles enjoying an infinitesimal
 characterisation which goes beyond the scope of quantum stochastic
 differential equations.
 \end{abstract}

\maketitle

\section*{Introduction}
\label{section: introduction}

The advent of the Hudson-Parthasarathy quantum stochastic calculus
coincided with the beginning of Arveson's study of product systems
of Hilbert spaces.
 The former is set in symmetric Fock space over an $L^2$-space of
 vector-valued functions (\cite{HuP}); the latter has the continuous tensor
 product decomposition of such Fock spaces
\begin{equation*}
%\label{decomp}
 \Fock = \Fock_{[0,t[} \ot \Fock_{[t,\infty[}, \quad
 t\in\Rplus,
\end{equation*}
 as its paradigm example (\cite{ArvMem}).
At a conference during the 1986-7 Warwick Symposium on Operator Algebras,
Arveson raised the following question.
 Noting that randomising a one-parameter unitary group
$(U_x)_{x\in\Real}$ on a Hilbert space $\init$ using a Brownian
motion $(B_t)_{t\geq 0}$:
\[
V_t \xi: \omega \to U_{B_t(\omega)}\xi(\omega),
 \quad t\in\Rplus, \xi\in\init\ot L^2(\Omega),
\]
defines a family of unitaries $V = (V_t)_{t\geq 0}$ on
 $L^2(\Omega; \init) = \init \ot L^2(\Omega)$
 satisfying a cocycle identity with respect to the shift on Brownian
 paths:
 \begin{equation*}
% \label{cocycle}
V_{r+t} = V_r \sigma_r(V_t), \quad r,t\in\Rplus,
 \end{equation*}
 he asked what other ways may such cocycles be generated --- are
 they all of this type?
 Armed with various quantum martingale representation theorems
 (\cite{HL non-Fock}, \cite{HLP}),
 the latter obtained in collaboration with Parthasarathy at the 1984-5 Warwick Symposium on
 Stochastic Differential Equations, Robin and I were able to provide an immediate answer.
 Our answer was a qualified `yes', if the process of randomisation is broadened to involve
% Brownian motion is replaced by
 \emph{quantum Brownian motion} (\cite{RobinQBM})
  along with the Poisson-type process derived from differential second quantisation,
  if adaptedness to the corresponding filtration of \emph{operator algebras} is imposed,
  and if one assumes sufficient regularity for the cocycle.
 The setting is \emph{Wiener space} $L^2(\Omega)$, with its
 semigroup of shifts inducing the semigroup
 $(\sigma_t)_{t\geq 0}$ on $B(L^2(\Omega))$
 now called \emph{CCR flow}. This is translated to symmetric
 Fock space $\Fock$ over $L^2(\Rplus)$ via the Wiener-Segal-It\^o
 isomorphism (see, for example~\cite{jml Greifswald}).

 \begin{thm}[\cite{HuL}]
 \label{HuL theorem}
 Let $V = (V_t)_{t\geq 0}$ be a unitary quantum
 stochastic cocycle on $\init$, and suppose that $V$ is Markov-regular.
 Then there is a unique operator
 $F
  \in B(\init \op \init)$
 such that $V$ satisfies the quantum stochastic differential equation
 \begin{equation*}
% \label{QSDE}
dV_t = V_t d\Lambda_F(t), \  V_0 = I_{\init\ot\Fock}.
  \end{equation*}
\end{thm}
(All the terms used here will be defined in Section~\ref{section:
QSC}). By the quantum It\^o formula, the coefficient $F$ necessarily
satisfies the Hudson-Parthasarathy unitarity conditions $q(F) = 0 =
r(F)$, where
\begin{equation*}
% \label{algebraic1}
q(F) := F^* + F + F^* \Delta F \text{ and }
 r(F):= F + F^* + F \Delta F^*,
\end{equation*}
 and $\Delta$ denotes the quantum It\^o projection
 $\left[\begin{smallmatrix}0 &  \\  & I_{\init}\end{smallmatrix}\right]
 \in B(\init \op \init)$.
 In terms of its block matrix form
% $F
% = \left[\begin{smallmatrix}K & M \\ L & C-I\end{smallmatrix}\right]$,
% where $H, L, C \in B(\init)$
 and the component quantum stochastic integrators,
 the unitarity conditions on $F$ read
\[
F = \begin{bmatrix} iH - L^*L/2 & -L^*C \\ L & C-I_{\init} \end{bmatrix},
\]
 with $C$ being unitary, $H$ selfadjoint and $L$ arbitrary, and 
 the quantum stochastic differential equation reads
\[
dV_t =  V_t\big( L\ot I_\Fock\, dA^* + (C-I_{\init})\ot I_\Fock\, dN_t  - 
 L^*C \ot I_\Fock\, dA_t + (iH- L^*L/2)\ot I_\Fock\, dt \big),
\]
 with
 $(A^*_t)_{t\geq 0}$, $(N_t)_{t\geq 0}$ and $(A_t)_{t\geq 0}$ being
 respectively the
 \emph{creation}, \emph{preservation}
 (\emph{number, exchange} or \emph{guage}) and \emph{annihilation
 processes}.

\begin{comment}
%\emph{Quantum stochastic cocycle} means family of operators
% $V = (V_t)_{t\geq 0}$ in $B(\init\ot\Fock)$ satisfying~\eqref{cocycle}
% and the adaptedness condition
% \[
% V_t \in B(\init) \otol B(\Fock_{[0,t[}) \ot I_{[t,\infty[}
%  \quad (t\geq 0)
% \]
%with respect to the continuous decomposition~\eqref{decomp};
%\emph{Markov-regular} means that the right-c
% semigroup of the
%cocycle
%\[
%(\mathbb{E}\circ V_t)_{t\geq 0},
%\]
% where $\mathbb{E} := \id_{B(\init)}\otol \omega_{\ve(0)}$,
% is norm-continuous.
% This result was proved as follows.
\end{comment}

 The result was proved as follows.
 Letting $K$ denote the generator of the expectation
 semigroup of $V$, $K$ is bounded (by Markov-regularity) and
 \[
 X := \Big( V_t - I_{\init\ot\Fock} - \int_0^t V_s (K\ot I_\Fock) ds \Big)_{t\geq 0}
 \]
defines a quantum martingale: $\mathbb{E}_s[X_t] = X_s$ ($s\leq t$),
satisfying $X_0 = 0$, which may be shown to be \emph{regular} in the
sense of Parthasarathy and Sinha (\cite{PaS})
 so that 
\[
 X_t = \Lambda_t(G) = 
 \int_0^t   L'_t \, dA^* + \int_0^t  (C'_t-I_{\init\ot\Fock})\, dN_t +  \int_0^t  M'_t \, dA_t ,
\]
% and therefore to be of the form $\big(\Lambda_t(G)\big)_{t\geq 0}$ 
where
$G = \Big(
\left[\begin{smallmatrix}
 0 & M'_t \\ L'_t & C'_t - I_{\init\ot\Fock}
 \end{smallmatrix}\right]\Big)_{t\geq 0}$
for bounded processes $L'$, $M'$ and $C'$.
 The proof is completed by verifying that each of the processes
 \[
(V_t^* L'_t)_{t\geq 0}, \ (V_t^* M'_t)_{t\geq 0} \text{ and } (V_t^*
C'_t)_{t\geq 0}
 \]
is (a.e.) constant.
% since the algebraic relations~\eqref{algebraic1}
%then follow from the Hudson-Parthasarathy unitarity theorem
%(\cite{HuP}).
 If the cocycle is instead adapted to the filtration of a non-minimal variance
 quantum Brownian motion then the same result holds (with no preservation integral)
 with similar proof, but using the martingale representation theorem
 for martingales with respect to this filtration (\cite{HL non-Fock}).

Journ\'e then gave a qualified `no' to Arveson's question.
 Along with an analysis of quantum stochastic cocycles $V$, assumed
 only to be strongly continuous, he gave an example to show that in
 general $V$ will \emph{not} satisfy a quantum stochastic
 differential equation.
 The essential point is that we may associate an operator $L$ to
 $V$ and, in favourable circumstances also an operator $C$,
 but the domains of $L$ and $L^*C$ may have insufficient
 intersection, thereby cheating us out of a dense domain for a 
 coefficient of a quantum stochastic
 differential equation.
 In his example $\init = L^2(\Rplus)$, $C=I_\init$ and $L$ is the generator of the right-shift
 semigroup on $\init$  (\cite{Journe}).

\begin{comment}%%%%%%%%%%%%%%%%%%%%%%%%%%%%
%Following the multidimensional extension of the above results
%(\cite{AJL}), two strands emerged in the subsequent analysis of
%quantum stochastic cocycles.
%
%Fagnola identified the condition
%
%\[
%F + F^* + F^* \Delta F \leq 0
%\]
%
%Working with finite particle domains Fagnola obtained the folowing
%result.
%
%\begin{thm}[\cite{Fagnola}]
%\label{Fagnola thm}
% Let $V$ be a strongly continuous quantum stochastic cocycle on $\init$ with noise
% dimension space $\kil$. Assume that $\init$ and $\noise$ are both separable,
% and $\domain$ is a core for the generator of its expectation
 %semigroup.
 %Then $V$ satisfies a quantum stochastic differential equation.
%\end{thm}
%
%***
%
%Mohari-Sinha to infinite dimensions
%
%***
%
%The second strand focused on the \emph{associated semigroups}
% $\{ P^{c,d}: c,d\in\noise\}$ of $V$. These are defined by
% \[
% \bip{u}{P^{c,d}v} =
% \bip{u\vp(c_{[0,t[})}{V_t v\vp(d_{[0,t[})}
% \quad (u,v\in\init, t\geq 0)
% \]
%in terms of the normalised exponential vectors
%\[
%\vp(g) := e^{-\norm{g}^2/2} \ve(g) \quad (g \in L^2(\Rplus;\noise)).
%\]
%
%******************************
\end{comment}

Bradshaw considered the corresponding question for quantum
stochastic mapping cocycles on a von Neumann algebra $\Nil$. Using
the quantum martingale representation theorem in a corresponding way
to the argument sketched above, he was able to show that every
Markov-regular, normal, injective, unital *-homomorphic cocycle on
$\Nil$ is governed by a quantum stochastic differential equation
(\cite{Bra}).

Subsequent developments in the construction and analysis of quantum
stochastic cocycles, up to the late nineties, are described
in~\cite{jml Greifswald} and~\cite{Fagnola}, both of which contain
extensive bibliographies.

 In this paper two new developments are described.
 The first, given in Section~\ref{section: Holder},
 is a direct and simple approach to differentiating
 mapping cocycles on an operator algebra, and more generally,
 on an operator space (\cite{jml Holder}).
 The basic assumption here is that there is an
 \emph{adjoint cocycle} and that both cocycles
 have locally H\"older continuous columns with exponent $1/2$.
 This analysis extends to quantum stochastic cocycles in Banach space and
 in abstract operator space (\cite{DasL}).
 The second, given in Section~\ref{section: holomorphic},
 is the case of contraction operator cocycles on a
 Hilbert space whose expectation semigroup is holomorphic
 (\cite{LSi holomorphic}).
 Here we go beyond quantum stochastic differential equations
 and yet still obtain a complete infinitesimal description of such
 cocycles.

 The definitions and basic properties
 of quantum stochastic cocycles, for both (contraction operator) cocycles on a
 Hilbert space and (mapping) cocycles on an operator space, are given
 in Section~\ref{section: QSC}. The latter
 requires \emph{matrix spaces} over an operator space (\cite{LW existence}); the basic facts
 about these are given in that section.

 \medskip
 \emph{Notations.}
 For a set $S$ and vector spaces $U$ and $V$ we write $F(S;V)$ for the
 linear space of functions from $S$ to $V$, under pointwise
 operations, and $L(U;V)$ for the space of linear maps from $U$ to $V$.
 For a vector-valued function $f: \Rplus \to V$ and subinterval $I$
 of $\Rplus$, $f_I$ denotes the function equal to $f$ on $I$ and
 zero outside $I$; for $c\in V$, $c_I: \Rplus \to V$ is defined in the same way, by
 viewing $c$ as a constant function.
 Simple tensors $u\ot \xi$ are usually abbreviated to $u\xi$.
 For a Hilbert space $\hil$ we set $\ket{\hil} = B(\Comp;\hil)$ and,
 mindful of the Riesz-Fr\'echet Theorem,
 $\bra{\hil} = B(\hil;\Comp)$, so that
 $\ket{\hil} = \{\ket{u}: u \in \hil\}$ and
 $\bra{\hil} = \{\bra{u}: u \in \hil\}$,
 where the -ket $\ket{u}$ maps $\lambda \in \Comp$ to $\lambda u$,
 and the -bra $\bra{u}$ maps $v\in\hil$ to $\ip{u}{v}$.
 The bra-/-ket notation is ampliated to the following
 useful $E$-\emph{notations}:
\begin{equation}
 \label{E-notations}
 E_\xi = I_\init \ot \ket {\xi}: \init \to \init \ot \hil, \text{ respectively }
 E^{\xi'} = I_\init \ot \bra {\xi}: \init'\ot\hil' \to \init'
\end{equation}
 for vectors $\xi$ and $\xi'$ from Hilbert spaces $\hil$ and $\hil'$.
 Ultraweak tensor products are denoted $\otol$ and purely algebraic
 tensor products by $\otul$. For an operator space $\Vil$ and
 Hilbert space $\hil$, $\iota_\hil^\Vil$ denotes the ampliation $x
 \in \Vil \mapsto x \ot I_\hil$.

 \medskip
\emph{Fix now, and for the rest of the paper, two Hilbert spaces}
 $\init$ \emph{and} $\noise$. Set $\khat:= \Comp \op \noise$ and, for
 $c\in\noise$, set $\chat:= \binom{1}{c} \in \khat$; also,
 for $f\in L^2(\Rplus;\noise)$ set $\wh{f}(t): = \wh{f(t)}$.
 We often make the identification
 \[
  \init \ot \khat =
 \init \op (\init\ot\noise),
 \]
 and employ the \emph{quantum It\^o projection}
 $\Delta :=
 \left[\begin{smallmatrix}0 &  \\  & I\end{smallmatrix}\right]
 \in B(\init\ot\khat) = B(\init \op (\init\ot \noise))$.

\section{Quantum stochastic cocycles}
 \label{section: QSC}

  For $0\leq r < t \leq \infty$, the
 symmetric Fock space over $L^2([r,t[;\noise)$ is denoted by $\Fock_{[r,t[}$
 and the identity operator on $\Fock_{[r,t[}$ by $I_{[r,t[}$;
 $\Fock_{\Rplus}$ is abbreviated to $\Fock$.
 We use normalised exponential vectors:
\[
 \vp(f) := e^{-\norm{f}^2/2} \ve(f),
 \quad
 f\in L^2(\Rplus;\noise),
\]
where $\ve(f)$ is the exponential vector $(1, f, (2!)^{-1/2}f^{\ot
2}, \cdots ) \in \Fock$.
 As is well-known, the family $\{\ve(f): f\in L^2(\Rplus; \noise)\}$ is linearly
 independent and total in $\Fock$. The following considerable
 strengthening of the latter property, due to Parthasarathy-Sunder for
 one dimensional $\noise$ and Skeide for general $\noise$, has
 proved very useful (see~\cite{jml Greifswald} for a proof). For a
 subset $S$ of $\noise$ define
 \[
 \Exps_S := \Lin \{ \ve(f): f \in \Step_S \}
 \]
 where
 $\Step_S$ :=
 \{$f \in L^2(\Rplus;\noise)$: $f$ is an $S$-valued step function\}
(with the convention that we always take the right-continuous
versions), and abbreviate $\Exps_\noise$ to $\Exps$.

 \begin{propn}
 \label{skeide}
 Let $\Til$ be a subset of $\noise$ containing $0$. Then the
 following are equivalent\tu{:}
 \begin{rlist}
 \item
 $\Til$ is total in $\noise$\tu{;}
 \item
 $\Exps_\Til$ is dense is $\Fock$.
 \end{rlist}
  \end{propn}

 Write $W(f)$ ($f\in L^2(\Rplus;\noise)$) for the unitary
 (\emph{Weyl}) operator on $\Fock$ determined by
 \[
  W(f) \vp(g) = e^{-i \im\ip{f}{g}} \vp(f+g), \quad
  g\in L^2(\Rplus;\noise),
 \]
$R_t$ ($t\in\Rplus$) for the unitary (\emph{time-reversal}) operator
on $\Fock$ determined by
\[
 R_t \vp(g) = \vp(r_tg) \text{ where }
 (r_t g)(s) = \begin{cases} g(t-s) & s \in [0,t[, \\ g(s) & s \in
 [t,\infty[, \end{cases}
  \quad
  (g\in L^2(\Rplus;\noise)),
\]
$S_t$ ($t\in\Rplus$) for the isometric (\emph{shift}) operator on
$\Fock$ determined by
\[
 S_t \vp(g) = \vp(s_t g) \text{ where }
 (s_t g)(r) = \begin{cases} 0 & r \in [0,t[, \\ g(r-t) & r \in
 [t,\infty[, \end{cases}
  \quad
  (g\in L^2(\Rplus;\noise)),
\]
$\sigma_t$ ($t\in\Rplus$) for the *-homomorphic (\emph{shift})
operator on $B(\init \ot \Fock)$ determined by
\[
 \bip{u\vp(f)}{\sigma_t(T)v\vp(g)} =
 \ip{\vp(f_{[0,t[})}{\vp(g_{[0,t[})}
 \bip{u\vp(s_t^*f)}{T\, v\vp(s^*_t g)},
\]
 ($T\in B(\init\ot\Fock)$, $f,g\in L^2(\Rplus;\noise)$, $u,v\in\init$),
and $\mathbb{E}$ for  the completely positive and contractive
(\emph{expectation}) map
\[
\mathbb{E} = \id_{B(\init)} \otol \omega_{\ve(0)} : B(\init\ot\Fock)
\to B(\init).
\]
By a \emph{quantum stochastic contraction cocycle on} $\init$
\emph{with noise dimension space} $\noise$, we mean a family of
contraction operators $V = (V_t)_{t\geq 0}$ on $\init\ot\Fock$
satisfying
\begin{align*}
 &s \mapsto V_s \text{ is strongly continuous}, \\
 &V_t \in B(\init\ot\Fock_{[0,t[}) \ot I_{[t,\infty[}, \\
 &V_{r+t} = V_r \sigma_r(V_t) \text{ and } V_0 = I_{\init\ot\Fock},
 \quad
 r,t\in\Rplus.
\end{align*}
 Let $\QScChk$ denote the collection of such cocycles.
 For $V\in\QScChk$,
 \[
 P^{0,0} :=
 \big(\mathbb{E}[V_t]\big)_{t\geq 0}
 \]
defines a contractive $C_0$-semigroup on $\init$ called the
\emph{expectation semigroup} of $V$.
 The semigroup property follows from the identities
\[
\mathbb{E} = \mathbb{E} \circ \mathbb{E}_r \ \text{ and } \
 \mathbb{E}[V_r(T\ot I_\Fock)] = \mathbb{E}[V_r] T,
 \quad
 r\in\Rplus, T \in B(\init),
\]
 for the conditional expectation maps
 \[
 \mathbb{E}_r: B(\init\ot\Fock) \to  B(\init\ot\Fock), \
 X \mapsto
 \big( \id_{ B( \init \ot \Fock_{[0,r[} ) } \otol \omega_{ \ve( 0_{ [r,\infty[} ) } \big) (X) \ot I_{[r,\infty[},
 \quad r\geq 0.
 \]
The continuity of $ P^{0,0}$ corresponds precisely to the strong (equivalently, weak operator)
 continuity assumption on $V$ (\cite{LW camb}, Lemma 1.2).
 The expectation semigroup is just one of the family
of \emph{associated semigroups} of $V$:
\[
\Pcd: =
 \big( (\id_{B(\init)} \otol
 \omega_{\vp(c_{[0,t[}),\vp(d_{[0,t[})})(V_t) \big)_{t\geq 0},
 \quad c,d\in\noise,
\]
 which together determine $V$ through the \emph{semigroup representation}
 (\cite{LW jfa}, Proposition 6.2).
 In fact any family $\{ P^{c,d}: c\in\Til', d\in\Til\}$
 where $\Til$ and $\Til'$ are each total in $\noise$ and contain $0$ suffices,
 thanks to Proposition~\ref{skeide}.
 This is important since, for example, if $\noise$ is one-dimensional
 then one may take $\Til = \{0,1\}$ and $V$ is then determined by
 its expectation semigroup together with just three of its other
 associated semigroups.

Write $\QScCMrhk$ for the subclass of \emph{Markov-regular} QS
contraction cocycles, that is those whose associated semigroups are
all norm-continuous. Contractivity of $V$ actually implies that
Markov-regularity is equivalent to norm-continuity of just the
expectation semigroup (\cite{LW jfa}, Theorem 6.6).

\begin{example}[Weyl cocycles]
 For $c\in\noise$,
 \[
W^c := \big( I_\init \ot  W(c_{[0,t[})\big)_{t\geq 0}
 \]
defines a Markov-regular QS contraction cocycle;
 $\mathbb{E}[W^c_t] = e^{-t\norm{c}^2/2} I_\init$.
\end{example}

 Two useful constructions of new cocycles from old are as follows.
 Let $V\in\QScChk$. Then the \emph{dual cocycle}  $\wt{V}$ is given
 by
\[
\wt{V}_t := \big( R_t V_t^* R_t \big)_{t\geq 0};
\]
 its associated semigroups are given, in terms of the associated
 semigroups $\{\Pcd: c,d\in\noise \}$ of $V$, by
 $\wt{P}_t^{c,d} = (P_t^{d,c})^*$; in particular, $\wt{V}$ is
 Markov-regular if and only if $V$ is. The \emph{associated
 cocycles} of $V$ are defined by
 \[
V^{c,d} := \big( (W^c_t)^* V_t W^d_t \big)_{t\geq 0},
 \quad c,d\in\noise.
 \]
 It is easily verified that these are indeed QS contraction cocycles; they
 are all Markov-regular if $V$ is, and the $(c,d)$-associated semigroup
 of $V$ is precisely the expectation semigroup of the $(c,d)$-associated
 cocycle of $V$.

\begin{example}
 Let $U = (U_x = e^{ixH})_{x\in\Real}$ be a strongly continuous
 one-parameter unitary group on $\init$ and let $(B_t)_{t\geq 0}$ be
 the standard Wiener process. Then
 \[
 (V_tF)(\omega) = U_{B_t(\omega)}F(\omega), \quad F \in
 L^2(\Omega;\init) = \init \ot L^2(\Omega),
 \]
 defines a family of unitaries $V = (V_t)_{t\geq 0}$ on
 $\init\ot L^2(\Omega)$ or, under the Wiener-Segal-It\^o
 isomorphism, on $\init\ot\Fock$ (with $\noise = \Comp$) which
 comprises a QS contraction cocycle with expectation semigroup
 $(e^{-tH^2/2})_{t\geq 0}$. This is the example highlighted by
 Arveson; it is Markov-regular if and only if $H$ is bounded, in
 other words $U$ is norm-continuous.
\end{example}

Processes other than Brownian motion may be used and multidimensions
can easily be incorporated. Here we are interested in the general
structure of subclasses of $\QScChk$ and this involves
\emph{quantum} processes, as we have seen. For many examples of
cocycles arising from quantum optics and classical probability
see~\cite{BHS},~\cite{Fagnola},~\cite{LW Indian}, and references
therein.

So far we have only considered contraction operator cocycles on a
Hilbert space. Given $V\in\QScChk$,
\begin{equation}
 \label{induced cocycle}
 k^V_t: x \in B(\init) \mapsto
 V_t (x \ot I_\Fock) V_t^* \in (\init \ot \Fock),
 \quad t\in\Rplus,
\end{equation}
defines a family of completely positive contractions $k^V =
(k^V_t)_{t\geq 0}$ enjoying the cocycle relation
\[
k_{r+t} = \wh{k}_r \circ \sigma_r \circ k_t, \quad r,t\in\Rplus,
\]
where $\wh{k}_r$ is the natural extension of $k_r$ to a map
\[
 \Ran \sigma_r = B(\init) \ot I_{[0,r[} \ot B(\Fock_{[r,\infty[}) \to
 B(\init) \ot B(\Fock_{[0,r[}) \ot B(\Fock_{[r,\infty[}) =
 B(\init \ot \Fock).
\]
Note that the \emph{induced cocycle} $k^V$ is unital if and only if
$V$ is coisometric, and is homomorphic if $V$ is partially isometric
with $V^*_tV_t \in I_\init \ot B(\Fock)$ ($t\geq 0$), in particular
if $V$ is isometric. Partially isometric and projection-valued 
 QS cocycles are analysed in [$\text{W}\!_{1,2}$].
 Note also that isometry for $V$ is equivalent to unitality for
 $k^{\wt{V}}$.

 In generalising the above class of cocycles, in particular to
 non-inner cocycles, it is convenient to drop the
 contractivity/boundedness condition. Fortunately the cocycle
 identity is easily expressible in terms of certain `slices' of the maps
 $(k_t)_{t\geq 0}$ as follows:
 \begin{equation}
  \label{weak cocycle}
 \kappa^{f,g}_{r+t} = \kappa^{f,g}_{r} \circ \kappa^{S^*_r f,S^*_r g}_{t},
 \quad r,t\in\Rplus,
 \end{equation}
 where $\kappa^{f,g}_{t}(x) :=
 \big(\id_{B(\init)}\otol \omega_{\xi,\eta}\big)\circ k_t$
 with $\xi=\vp(f_{[0,t[})$ and $\eta = \vp(g_{[0,t[})$,
 and this requires only that each $k_t(x)$ is an operator whose
 domain includes sufficiently many vectors of the form $u\vp(h)$ and
 that the resulting operators $\kappa^{f,g}_{t}(x)$ are bounded.

 A \emph{concrete operator space} is a closed subspace $\Vil$ of $B(\Hil;\Kil)$
 for some Hilbert spaces $\Hil$ and $\Kil$;
 we speak of $\Vil$ being an operator space \emph{in} $B(\Hil;\Kil)$, or
 $B(\Hil;\Kil)$ being the \emph{ambient} (\emph{full operator}) \emph{space} of $\Vil$.
 The adjoint operator space $\Vil^\dagger$  is the operator space $\{x^*: x \in \Vil\}$
 in  $B(\Kil;\Hil)$, and the
 adjoint map of a linear map between operator spaces $\phi:
 \Vil\to\Wil$ is the map $\phi^\dagger: \Vil^\dagger \to
 \Wil^\dagger$, $x^* \mapsto \phi(x)^*$.
 A basic notion is that of \emph{complete boundedness}
 for a linear map between operator spaces  $\phi:\Vil\to\Wil$.
 This means that the norms of the matrix liftings
 \begin{equation}
 \label{lifting}
 \phi^{(n)}:  [v_{i,j}] \mapsto [\phi(v_{i,j})],
 \end{equation}
 in which $M_n(\Vil)$ has the norm inherited from $B(\Hil;\Kil)$,
 are uniformly bounded:
\[
 \cbnorm{\phi} := \sup_{n\in\Nat} \norm{ \phi^{(n)} } < \infty.
\]
There are several excellent texts on operator space theory
(\cite{BlM}, \cite{EfR}, \cite{Pau}, \cite{Pis}).

 Now let $\Vil$ be an operator space in $B(\init;\init')$.
 For any Hilbert spaces $\hil$ and $\hil'$, the $\hil$-$\hil'$
 \emph{matrix space over} $\Vil$ is defined as follows:
 \[
 \Vil \otM B(\hil;\hil') :=
 \big\{ T \in B(\init\ot\hil;\init'\ot\hil'):
 \forall_{\omega\in B(\hil;\hil')_*}\
 (\id_{B(\init;\init')}\otol\omega)(T)\in \Vil \big\}
 \]
 (\cite{LW existence}).
 Recalling the $E$-notations~\eqref{E-notations}, a
 convenient characterisation, in terms of any total subsets $S$
 and $S'$ of $\hil$ and $\hil'$, is as follows:
\[
 \Vil \otM B(\hil;\hil') =
 \big\{ T \in B(\init\ot\hil;\init'\ot\hil'):
 \forall_{\xi\in S, \xi'\in S'}\
 E^{\xi'} T E_{\xi} \in \Vil \big\}.
 \]
 Thus $\Vil \otM B(\hil;\hil')$ is an operator space in
 $B(\init\ot\hil;\init'\ot\hil')$.

 \emph{Warning}.\
 For a $\Cstar$-algebra $\Al$, the operator space $\Al\otM B(\init)$
 \emph{need not} be a $\Cstar$-algebra (see~\cite{LW existence} for an example).

 The lifting of maps $\phi$ between operator spaces to maps between
 matrices over the operator spaces~\eqref{lifting}
 extends to matrix spaces as follows. Let $\phi \in CB(\Vil;\Wil)$
 for concrete operator spaces $\Vil$ and $\Wil$, then there is a
 unique map
 \[
 \phi \otM \id_{B(\hil;\hil')} :
 \Vil \otM B(\hil;\hil') \to \Wil \otM B(\hil;\hil')
 \]
satisfying
 \[
 E^{\xi'} (\phi \otM \id_{B(\hil;\hil')}\big)(T) E_\xi =
 \phi(E^{\xi'} T E_\xi),
 \quad
 T\in \Vil\otM B(\hil;\hil'), \xi\in\hil, \xi'\in\hil',
 \]
 moreover, $\phi \otM \id_{B(\hil;\hil')}$ is a completely bounded
 operator with cb-norm at most $\cbnorm{\phi}$
 (\cite{LW existence}).
 If $B(\hil;\hil')$ is finite dimensional then $\phi$ need only be
 bounded for $\phi \otM \id_{B(\hil;\hil')}$ to exist; in this case
 the matrix lifting is a bounded operator.
 To handle both situations conveniently we coined the term
 $\hil$-\emph{boundedness} for a linear map $\phi:\Vil\to\Wil$,
 meaning bounded/completely bounded according as $\hil$ is
 finite-dimensional/infinite-dimensional, and write
 $\hil$-$B(\Vil;\Wil)$ for this class of maps.

\begin{term} % a final (post Kuo) edit insert
A \emph{QS cocycle on} $\Vil$ \emph{with noise dimension space}
 $\noise$ \emph{and exponential domain} $\Exps_\Til$ (where $\Til$
 is a total subset of $\noise$ containing $0$)
  is a family $k = (k_t)_{t\geq 0}$ of linear maps
\[
 k_t : \Vil \to L(\init \otul \Exps_\Til; \init' \ot\Fock)
 \cong
  L\big( \Exps_\Til; L( \init;\init' \ot\Fock ) \big)
\]
satisfying
\[
 E^{\ve(f)} k_t(x) E_{\ve(g)} \in \Vil,
% \quad
% (x\in\Vil, \ve \in\Exps_\Til, \xi'\in\Fock, t\in\Rplus),
\]
the \emph{weak cocycle relation}~\eqref{weak cocycle} and the
\emph{adaptedness condition}
\[
 \ip{u'\ve(f)}{k_t(x)u\ve(g)} =
 \ip{\ve(f_{[t,\infty[})}{\ve(g_{[t,\infty[})}
 \ip{u'\ve(f_{[0,t[})}{k_t(x)u\ve(g_{[0,t[})}
\]
 ($x\in\Vil$, $u\in\init$, $u'\in\init'$, $g\in\Step_\Til, f\in \Step$, $t\in\Rplus$).
%\end{term}
%
 The collection of such cocycles is denoted $\QSCVET$.
%\end{term} 
%(We shall
% impose the appropriate continuity condition in $t$ shortly).
 The subclass of cocycles $k\in\QSCVET$ having an \emph{adjoint
 cocycle} $k^\dagger\in\QSCVdaggerETprime$, so that
 \[
 E^{\ve} k^\dagger(x^*)E_{\ve'} = (E^{\ve'} k(x)E_{\ve})^*,
 \quad
 x\in\Vil, \ve\in\Exps_\Til, \ve' \in \Exps_{\Til'},t\in\Rplus,
 \]
is denoted $\QSCddaggerVETprimeET$.
\end{term}
 A cocycle $k$ has $\hil$-\emph{bounded columns} if, in the notation
 $k_{t,\ve} := k_t(\cdot)E_\ve: \Vil\to L(\init; \init'\ot\Fock)$, it satisfies
% \begin{alist}
 \begin{itemize}
\item
 $k_{t,\ve}(\Vil)
 %\subset B(\init;\init'\ot\Fock)$,
 %which implies that $k_{t,\ve}(\Vil)
 \subset \Vil\otM\ket{\Fock}$, and
 \item
 $k_{t,\ve}$ is $\hil$-bounded $\Vil \to \Vil\otM\ket{\Fock}$, \
 $t\in\Rplus, \ve\in\Exps_\Til$.
% \end{alist}
  \end{itemize}
 We say that a cocycle $k$ is $\hil$-\emph{bounded} if it satisfies the stronger conditions
% \begin{alist}
 \begin{itemize}
\item
 %$k_{t}(\Vil) \subset B(\init\ot\Fock;\init'\ot\Fock)$,
 %which implies that
 $k_{t}(\Vil) \subset \Vil\otM B(\Fock)$, and
\item
 $k_{t}$  is  $\hil$-bounded $\Vil \to \Vil\otM B(\Fock),
 \ t\in\Rplus$.
% \end{alist}
 \end{itemize}

 Mapping cocycles $k\in \QSCVET$ also have associated semigroups:
 \[
 \calPcd := \big( E^{\vp(c_{[0,t[})} k_t(\cdot) E_{\vp(d_{[0,t[})}
 \big)_{t\geq 0},
 \quad
 c\in\noise, d\in\Til,
 \]
 and, as for operator cocycles, the collection of associated
 semigroups determines the cocycle --- thanks to
 Proposition~\ref{skeide}.
 For a thorough investigation of the reconstructability of cocycles
 from compatible families of semigroups
 see ~\cite{LW cb1}, which was inspired by~ \cite{AcK}.

 We now introduce the continuity condition which plays the central
 role in Section~\ref{section: Holder}.
 Let $\QSChbHcVET$ denote the class of cocycles $k\in \QSCVET$
 having $\hil$-bounded columns and such that
 \[
 t \mapsto k_{t,\ve}
 \text{ is H\"older }\frac{1}{2}\text{-continuous }
 \Rplus \to \hil\text{-}B(\Vil;\Vil\otM \ket{\Fock}) \text{ at } t=0,
 \]
and let $\QSCddaggerhbHcVETprimeET$ denote the subclass
\[
\Big\{ k \in \QSCddaggerVETprimeET \cap \QSChbHcVET : k^\dagger \in
\QSChbHcVdaggerETprime \Big\}.
\]
 We refer to elements of  $\QSCkhatbHcVET$  as \emph{H\"older cocycles}.
 The reason for highlighting this class of cocycle
 is hinted at in the observation~\eqref{further property}
 and fully justified in Theorem~\ref{four bijections}.
 In  this connection, note the elementary estimate
 \begin{equation}
 \label{Weyl Holder}
 \norm{ W^c_{t,\ve(0)} - W^c_{r,\ve(0)} }  =
 \norm{c} |t-r|^{1/2} + O(t-r) \text{ as } t\to r,
 \quad
 c\in\noise.
\end{equation}
% as $t\to r$ in a bounded interval.

We again call cocycles $k \in \QSCVET$ \emph{Markov-regular}
 (respectively, \emph{cb-Markov-regular}) when each associated
 semigroup is norm-continuous (respectively,  cb-norm-continuous),
 and we have the inclusions:
 \[
 \QSChbHcVET \subset \QSCMrVET, \text{ and }
 \QSCcbHcVET \subset \QSCcbMrVET.
 \]
%\begin{rems}
Note that 
 the prescription
\vspace{-0.1cm}
\begin{equation}
 \label{correspondences}
 k_t(\ket{u}) =
 V_t \big( \ket{u}\ot I_\Fock \big)
\end{equation}
delivers a healthy traffic between the completely contractive
mapping cocycles
$k\in\mathbb{QSC}^\ddagger(\ket{\init}:\Exps,\Exps)$ and the
contraction operator cocycles $V\in\QScChk$.
 The respective associated semigroups are related by
 \[
  \calPcd_t (\ket{u}) = \ket{\Pcd_t u} =  \Pcd_t\ket{ u},
 \quad
 c,d\in\noise, t\in\Rplus.
\]
The basic contraction mapping cocycles defined by
\[
 k_t(\ket{u}) = W^c_t \big( \ket{u}\ot I_\Fock \big)
\]
lie in
$\mathbb{QSC}^\ddagger_{\mathrm{cbHc}}(\ket{\init}:\Exps,\Exps)$,
 as follows from~\eqref{Weyl Holder}.
%\end{rems}

%%%%%%%%%%%%%%%%%%%%%%%%% NEW SECTION %%%%%%%%%%%%%%%%%%%%%%%%%%%%
\section{H\"older continuous cocycles}
 \label{section: Holder}

\emph{For this section we fix an operator space} $\Vil$ \emph{in}
$B(\init;\init')$. We first discuss the generation of QS cocycles on
$\Vil$ by means of QS differential equations, and then we
characterise various classes of cocycle so-generated. The main
result (Theorem~\ref{four bijections}) is the converse to the
existence theorem established in~\cite{LW existence}.
 Further details and full proofs
 will apear in the forthcoming paper~\cite{jml Holder}.

 Let $\varphi \in
 F(\{1_\Comp\}\times\Til;
 \khat\text{-}B(\Vil; \Vil\otM \ket{\khat})$
 for a total subset $\Til$ of $\noise$ containing $0$.
 Then, for any total subset $\Til'$ of $\noise$ containing $0$, the
 QS differential equation

\begin{equation}
\label{QSDEmapping}
 dk_t = k_t \circ d \Lambda_\varphi(t)
 \quad
 k_0 = \iota^\Vil_\Fock,
\end{equation}
has a unique
 $(\Exps_{\Til'}, \Exps_{\Til})$-weakly regular, $(\Exps_{\Til'}, \Exps_{\Til})$-weak
 solution, denoted $k^\varphi$.

%\begin{rem}
 Usually coefficients of QS differential equations are assumed to be
 \emph{linear}:
 $\phi \in L(\wh{\Dil}; \khat\text{-}B(\Vil;\Vil\otM \ket{\khat})$
 for a dense subspace $\Dil$ of $\noise$. However, bare QS differentiation of cocycles
 only yields coefficients $\varphi$ which are functions defined on
 $\{1_\Comp\}\times\Til := \{\chat : c\in\Til\}$.
 As we shall see, linearity is recovered under the assumption that
 the cocycle also has a H\"older adjoint cocycle.
%\end{rem}

 Here
 $(\Exps_{\Til'}, \Exps_{\Til})$-\emph{weak solution} means
 \begin{align*}
 &k_t \in L\big(\Vil; L(\init \otul \Exps_\Til;\init'\ot\Fock)\big),
 E^{\ve(f)} k_t(x) E_{\ve(g)} \in \Vil, \\
 &s \mapsto \ip{u'\ve}{k_s(x) u\ve} \text{ is continuous, and } \\
 &\ip{u'\ve(f)}{k_t(x)u\ve(g)} =
 \ip{u'}{xu}\ip{\ve(f)}{\ve(g)} + \\
 & \qquad \qquad \qquad \qquad \qquad \qquad \qquad
 \int_0^t ds\ \big\la u'\ve(f),
 k_s\big(E^{\wh{f}(s)}\varphi_{\wh{g}(s)}(x)\big)u\ve(g) \big\ra
 \end{align*}
($u'\in\init', f\in\Step_{\Til'}, x\in \Vil, u\in\init,
 g\in\Step_\Til$);
 $(\Exps_{\Til'}, \Exps_{\Til})$-\emph{weakly regular} means
\[
 s \mapsto E^{\ve'}k_s(\cdot)E_\ve \text{ is locally bounded }
 \Rplus \to B(\Vil),
 \quad
 \ve \in \Exps_\Til, \ve' \in \Exps_{\Til'}.
\]

 The unique solution enjoys the following further property:
 \begin{equation}
 \label{further property}
 k^\varphi \in \QSCkhatbHcVET,
 \end{equation}
 and if
 $\varphi(\{1_\Comp\}\times\Til) \subset
 CB(\Vil;\Vil\otM \ket{\khat})$
 then
$ k^\varphi \in \QSCcbHcVET$.
 Moreover, if $\varphi$ has an adjoint map
 $\varphi^\dagger \in
 F\big(\{1_\Comp\}\times \Til';
 \khat\text{-}B(\Vil^\dagger; \Vil^\dagger\otM \ket{\khat})\big)$
 satisfying
 \[
 E^{\chat} \varphi^\dagger_{\dhat} (x^*) =
 \big(E^{\dhat} \varphi_{\chat} (x)\big)^*,
 \quad
  x\in\Vil, c\in\Til, d\in\Til',
 \]
 then $k^\varphi \in \QSCddaggerkhatbHcVETprimeET$
 with adjoint cocycle $k^{\varphi^\dagger}$.
 In this case,
 \[
 \varphi = \phi|_{\{1_\Comp\}\times\Til}
 \]
 for a unique map $\phi$ in
 $L^\ddagger_{\wh{\Dil'}}\big(\wh{\Dil};
 \khat\text{-}B(\Vil;\Vil\otM\ket{\khat})\big)$,
 the space of linear maps
 $\phi : \wh{\Dil} \to \khat\text{-}B(\Vil;\Vil\otM\ket{\khat})$
 having an adjoint linear map
 $\phi^\dagger : \wh{\Dil'} \to
 \khat\text{-}B(\Vil^\dagger;\Vil^\dagger\otM\ket{\khat})$,
 where $\wh{\Dil} = \Comp \op \Dil$, $\Dil = \Lin \Til$
 and similarly for $\Dil'$ and $\Til'$,
 moreover
 \[
 k^\varphi_t = k^\phi_t|_{\Exps_\Til},
 \quad
 t\in\Rplus.
 \]
 The above facts are essentially contained in~\cite{LW existence},
 supplemented by~\cite{LSk qsde};
 the minor modifications needed for the present generality are
 explained in~\cite{jml Holder}.

 The situation is summarised as follows.
 The map
 \[
 F\big(\{1_\Comp\}\times \Til;
 \khat\text{-}B(\Vil; \Vil\otM \ket{\khat})\big) \to
 \QSCkhatbHcVET, \ \varphi \mapsto k^\varphi,
 \]
is injective and restricts to maps
\begin{align*}
 &F\big(\{1_\Comp\}\times \Til;
 CB(\Vil; \Vil\otM \ket{\khat})\big) \to
 \QSCcbHcVET \\
 &L^\ddagger_{\wh{\Dil'}}\big(\{\wh{\Dil};
 \khat\text{-}B(\Vil; \Vil\otM \ket{\khat})\big) \to
 \QSCddaggerkhatbHcVETprimeET \\
 &L^\ddagger_{\wh{\Dil'}}\big(\{\wh{\Dil};
 CB(\Vil; \Vil\otM \ket{\khat})\big) \to
 \QSCddaggercbHcVETprimeET,
\end{align*}
where $\Dil = \Lin \Til$ and $\Dil' = \Lin \Til'$.

 The following theorem extends Theorem 5.6 of~\cite{LSk qsde},
 where $\Vil$ is assumed to be finite dimensional.

 \begin{thm}
 \label{four bijections}
 Each of the above four maps is bijective.
 \end{thm}

 The idea of the proof is to differentiate directly, in other words,
 to show that for $k \in \QSCkhatbHcVET$ and $c\in\Til$, the family
 \[
 \varphi_{\chat,t} :=
 t^{-1} \begin{bmatrix} E^{\ve(0)} \\ E^{1_{[0,t[}} \end{bmatrix}
 \big(k_{t,\ve(c_{[0,t[})} - \iota^\Vil_{\ve(c_{[0,t[})}\big)(\cdot)
 \quad
(t>0)
 \]
  converges
 in $\khat$-$B(\Vil;\Vil\otM\ket{\khat})$
 to a map $\varphi_{\chat}$, as $t\to 0^+$,
 in the point-W.O. topology:
 \[
 \ip{\xi}{\varphi_{\chat ,t}(x)\eta} \to
  \ip{\xi}{\varphi_{\chat}(x)\eta},
 \quad
 x\in\Vil, \xi\in\init', \eta\in\init,
 \]
 and then to show that the resulting map
 $\varphi \in F\big(\{1_\Comp\}\times \Til;
 \khat\text{-}B(\Vil; \Vil\otM \ket{\khat})\big)$ satisfies
 $k = k^\varphi$ by verifying that the two (Markov-regular) cocycles
 $k$ and $k^\varphi$
 have the same associated semigroup generators. The rest then follows
 easily. Here
  \begin{multline}
  \label{E 1 0,t}
 E^{1_{[0,t[}} :=
 I_{\init'\ot\noise}\ot \bra{1_{[0,t[}} =
  I_{\init'} \ot \big( I_\noise \ot \bra{1_{[0,t[}} \big)
 \\
 \in
 B(\init'\ot\Fock;\init'\ot\noise) =
  B(\init') \otol B(\Fock;\noise),
 \end{multline}
%  \begin{equation}
 %\label{E 1 0,t}
% E^{1_{[0,t[}} :=
% I_{\init'\ot\noise}\ot \bra{1_{[0,t[}} =
%  I_{\init'} \ot \big( I_\noise \ot \bra{1_{[0,t[}} \big)
% \in
% B(\init'\ot\Fock;\init'\ot\noise) =
%  B(\init') \otol B(\Fock;\noise),
% \end{equation}
 by means of the inclusion
 $\noise \ot L^2(\Rplus) = L^2(\Rplus;\noise) \subset \Fock$, as
 one-particle subspace, and $\iota^\Vil_\ve(x) := x\ot\ket{\ve}$ for $\ve\in\Exps_\Til$,
 $x\in\Vil$.

 For bounded QS cocycles we can say more. Note first the natural
 inclusions
 \[
 \hil\text{-}B(\Vil;\Vil\otM B(\khat)) \subset
 L\big(\khat; \hil\text{-}B(\Vil;\Vil\otM\ket{\khat})\big),
 \]
for any Hilbert space $\hil$, arising from the identification
 $\phi_{\chat} = \phi(\cdot)E_{\chat}$ ($c\in\noise$).

\begin{thm}
 Let $k \in \QSCddaggerkhatbHcVETprimeET$ be bounded,
 with locally uniform bounds.
 Then $k = k^\psi|_{\Exps_\Til}$ for a unique map
 $\psi \in \khat$-$B(\Vil;\Vil\otM B(\khat))$.
 Moreover, if $k \in \QSCddaggercbHcVETprimeET$ then
 $\psi \in CB(\Vil;\Vil\otM B(\khat))$.
\end{thm}

 The idea of the proof of this is to let
 $\phi \in L^\ddagger_{\wh{\Dil'}}\big(\wh{\Dil};
 \khat$-$B(\Vil;\Vil\otM\ket{\khat})\big)$ be the mapping arising
 from Theorem~\ref{four bijections},
 to set
 \[
 \tau = E^{\wh{0}}\phi_{\wh{0}}(\cdot), \
 \chi = \Delta \phi_{\wh{0}}(\cdot) \text{ and }
 \alpha = \big( \Delta \phi^\dagger_{\wh{0}}(\cdot)\big)^\dagger
 \]
and to establish convergence, as $t\to 0^+$, of the (locally
uniformly bounded) family
\[
 \nu_t := t^{-1} E^{1_{[0,t[}} k_t(\cdot) E_{1_{[0,t[}}
 \quad
 (t>0)
\]
to a map $\nu \in B(\Vil;\Vil\otM B(\noise))$, in the point-W.O.
topology. The resulting map
 $\psi :=
 \left[\begin{smallmatrix}
 \tau & \alpha \\ \chi  & \nu - \iota
 \end{smallmatrix}\right]$,
 where $\iota = \iota^\Vil_\noise$,
 lies in $B(\Vil;\Vil\otM B(\khat))$ and satisfies the identity
 \[
 E^{\chat} \psi(\cdot)E_{\dhat} =
 E^{\chat} \phi_{\dhat}(\cdot),
 \quad
 d\in\Dil, c\in\Dil';
 \]
the rest then follows easily.

 The following two known results ([$\text{LW}\!_{3,2}$])
 % (\cite{LW existence},~\cite{LW jfa})
 may be easily deduced from the above theorem.

 \begin{cor}
 Let $k\in\QSCMrAE$ be completely positive and contractive, on a
 $\Cstar$-algebra $\Al$.
 Then there is $\phi \in CB(\Al; \Al\otM B(\khat))$ such that
 $k=k^\phi$.
 \end{cor}

This is proved by judicious use of the operator Schwarz inequality
which reveals that, under the given hypotheses, $k\in\QSCddaggercbHcAEE$.

%\begin{rem}
The precise form that $\phi$ must take for $k^\phi$ to be completely
positive and contractive is given in~\cite{LW jfa}; this is a
stochastic extension of the Christensen-Evans Theorem (\cite{ChE}).
A new proof of the form of $\phi$,  using global Schur-action semigroups on matrix spaces, is
given in~\cite{LW cb2}.
%\end{rem}

\begin{cor}
 Let $V\in\QScCMrhk$.
 Then there is a unique operator
 $F\in B(\init\ot\khat)$ such that $V = V^F$, that is $V$ is the
 unique \tu{(}weak, contractive\tu{)} solution of the QS differential equation
 $dV_t = V_t d\Lambda_F(t)$, $V_0 = I_{\init\ot\Fock}$.
\end{cor}

This is deduced from Theorem~\ref{four bijections} via the
correspondences~\eqref{correspondences}.

% a final (post Kuo) edit insert - separate paragraph
%\begin{rem}
The quantum It\^o formula provides the following necessary and
sufficient condition on an operator $F\in B(\init\ot\khat)$
for $V^F$ to be contractive:
\begin{equation}
 \label{q(F)}
q(F) \leq 0 \text{ where } q(F) := F^* + F + F^*\Delta F,
\end{equation}
equivalently, $r(F) \leq 0$ where $r(F) := F + F^* + F\Delta F^*$
 (\cite{LW ptrf}, Theorem 7.5).
 Thus $F \mapsto V^F$ defines a bijection
 \begin{equation}
 \label{Mr bijection}
 C_0(\init,\noise):= \{F \in B(\init\ot\khat): q(F) \leq 0 \} \to
 \QScCMrhk,
 \end{equation}
extending the bijection
 \[
 \{F \in B(\init\ot\khat): q(F) = 0 = r(F) \} \to
 \QSuCMrhk
 \]
 given in Theorem~\ref{HuL theorem}
 (in case $\noise = \Comp$), % a final (post Kuo) edit insert
 where
 $ \QSuCMrhk$ denotes the collection of
 Markov-regular \emph{unitary} cocycles.
%\end{rem}

In the following section we describe a significant extension of this
result, to a class of
 \emph{strongly continuous} operator cocycles, which
necessarily goes beyond the realm of QS differential equations.

\section{Holomorphic cocycles}
 \label{section: holomorphic}

 \emph{For this section we again fix Hilbert spaces} $\init$ \emph{and} $\noise$.
 We introduce the class of holomorphic QS contraction cocycles on
 $\init$ with noise dimension space $\noise$, and consider their
 infinitesimal characterisation. We then describe the connection to
 minimal quantum dynamical semigroups. The main results of this section are
 Theorems~\ref{holomorphic bijection} and~\ref{underpins}.
 Further details and full proofs
 will appear in the forthcoming paper~\cite{LSi holomorphic}.

 Let $\Xh$ denote the class of operators $K$ on $\init$ which are
 densely defined and dissipative:
 $\re\ip{u}{Ku} \leq 0$ ($u\in\Dom K$)
 with no dissipative extension, in other words \emph{maximal dissipative}.
 It is well-known that such operators are precisely the generators
 of contractive $C_0$-semigroups on $\init$
 (\cite{Dav}, Theorem 6.4).
 We are interested in the collection $\Xholh$ of quadratic forms $q$
 on $\init$ which are \emph{accretive} and \emph{semisectorial}:
 \[
 \re q[u] \geq 0 \text{ and } \big|\im q[u]\big| \leq C \|u\|_+^2,
 \quad
  u\in \calQ,
 \]
for some $C\geq 0$, where
\[
 \calQ := \Dom q \text{ and }
 \|u\|_+^2 := \big(\re q[u] + \|u\|^2\big)^{1/2},
\]
as well as being \emph{densely defined} and \emph{closed}:
\[
 \calQ \text{ is dense in } \init \text{ and }
 \calQ \text{ is complete w.r.t. the norm } \| \cdot \|_+.
\]
% \begin{rem}
 The term ``semisectorial'' is nonstandard, but avoids the potentially confusing term
 ``continuous'' favoured by some.
``Sectorial'', which is standard, is the strengthening
 $|\im q[u]| \leq C \re q[u]$  ($ u\in \calQ$) for some $C\geq 0$.
% \end{rem}

For $q \in \Xholh$ there is a unique operator $K$ on $\init$
satisfying
\begin{align*}
 &\Dom K = \Big\{ u \in \calQ: \text{
 the (conjugate-linear) functional }
 v\in\calQ \mapsto q(v,u) \text{ is bounded} \Big\}, \\
 &\ip{u}{Ku} = - q[u], \quad u\in\calQ,
\end{align*}
where $\calQ = \Dom q$ and $q(\ ,\ )$ is the sesquilinear form
associated with $q[\cdot]$ via polarisation, moreover $K \in \Xh$.
 The operators $K\in\Xh$ that arise from forms $q\in\Xholh$ are
 precisely those generators which are semisectorial:
  $|\im \ip{u}{Ku}| \leq C
 \big(-\re \ip{u}{Ku} + \|u\|^2\big)^{1/2}$
 ($u\in \Dom K$) for some $C\geq 0$.
 In this way we view $\Xholh$ as a subset of $\Xh$.

By a \emph{holomorphic contraction semigroup} on $\init$ we mean a
contractive $C_0$-semigroup $P = (P_t)_{t\geq 0}$ for which there is an angle
$\theta\in ]0,\pi/2]$ such that $(e^{-t}P_t)_{t\geq 0}$ extends to a
contraction-valued holomorphic function $\Sigma_\theta \to
B(\init)$, where $\Sigma_\theta$ denotes the open sector of the complex
plane $\{ z \in \Comp\setminus\{0\}: |\arg z| < \theta \}$.

 \emph{Warning}.
Definitions in the literature vary.
 With our definition, the holomorphic contraction semigroups are
 precisely those whose generators are in $\Xholh$
 (see e.g.~\cite{Ouh}).

 \begin{defn}
 A cocycle $V\in\QScChk$ is \emph{holomorphic} if its
 expectation semigroup is holomorphic. Write $\QScCholhk$ for
 the resulting class of cocycles.
 \end{defn}
 For $V \in \QScCholhk$ we write $\gV$ for the form-generator
 of its expectation semigroup and $\calQ^V$ for $\Dom \gamma^V$.
 The dual of a holomorphic QS contraction cocycle $V$ is holomorphic
 with
 \[
 \gamma^{\wt{V}} = (\gV)^*:= u \mapsto \ol{\gV[u]}
 \text{ and } \calQ^{\wt{V}} = \calQ^V.
 \]
Since norm-continuous semigroups are holomorphic, we have the
inclusion
\begin{equation}
 \label{hol contains Mr}
\QScCholhk \supset \QScCMrhk.
\end{equation}

 Kato--type relative boundedness arguments for quadratic forms (\cite{Kato})
yield our first two consequences of the holomorphic assumption.

\begin{propn}
 \label{3.1*}
Let $V\in\QScCholhk$. Then the following hold\tu{:}
\begin{alist}
\item
 All of the associated semigroups of $V$ are holomorphic and their
 form-generators have equal domain $\calQ^V$.
\item
$V$ is \emph{nonsingular}, that is, with
 $E_{1_{[0,t[}} :=
 \big( E^{1_{[0,t[}} \big)^* \in B(\init\ot\noise;\init\ot\Fock)$,
\[
t^{-1} E^{1_{[0,t[}} V_t E_{1_{[0,t[}} \to C^V
 \text{ in the weak operator topology as } t\to 0^+,
\]
for a contraction operator $C^V$ on $\init\ot\noise$.
\end{alist}
\end{propn}

%\begin{rem}
 Note that, 
 in view of the identity $R_t E_{1_{[0,t[}} = E_{1_{[0,t[}}$,
 it follows that, for
 $V \in \QScCholhk$,
 \[
 C^{\wt{V}} = (\CV)^*.
 \]
%\end{rem}

 We come now to the central definition of this section. 

\begin{notn}
Let
 $\Xfourholhk$ denote the collection of quadruples
$\mathbb{F} = (\gamma, L, \Ltilde, C-I)$, where $\gamma \in \Xholh$,
$L$ and $\Ltilde$ are operators from $\init$ to $\init \ot \noise$
with domain $\calQ: = \Dom \gamma$ and $C$ is a contraction in
$B(\init\ot\noise)$, satisfying
\begin{equation}
 \label{structure inequality}
 \| \Delta F \zeta \|^2 \leq 2 \re \Gamma[\zeta]
\end{equation}
 for the operator $\Delta F$ and
 quadratic form $\Gamma$ on $\init \op (\init\ot\noise) = \init \ot\khat$
 given by
 \begin{align*}
 &\Dom \Gamma =\ \Dom \Delta F := \calQ \op (\init\ot\noise),\
 \Delta F := \begin{bmatrix} 0 & 0 \\ L & C-I \end{bmatrix}, \text{ and }  \\
 &\Gamma[\zeta] :=
 \gamma[u] -
 \big( \ip{\xi}{Lu} + \ip{\Ltilde u}{\xi} +
 \ip{\xi}{(C-I)\xi} \big)
 \ \text{ for }
 \zeta = \binom{u}{\xi} \in \Dom \Gamma.
 \end{align*}
\end{notn}
\begin{remarks}
\label{cited Remark}
 (i)
 The relation
 \eqref{structure inequality} contains the inequalities
\begin{equation}
 \label{minimal inequality}
 \| Lu \|^2 \leq 2 \re \gamma[u], \
 \| \Ltilde u \|^2 \leq 2 \re \gamma[u],
 \quad u \in \calQ,
\end{equation}

 (ii)
 If $\gamma$ is bounded, so that $\calQ = \init$, then $L$ and
 $\Ltilde$ are bounded, $ \Gamma[\zeta] = - \ip{\zeta}{F\zeta}$
 ($\zeta \in \init\ot\khat$)
 where
  $F :=
 \left[\begin{smallmatrix} K & M \\ L & C-I
 \end{smallmatrix}\right]$
 in which $M = \Ltilde^*$  and $K$ is the bounded operator
 associated to $\gamma$: $\ip{u}{Ku} = -q[u]$ ($u\in\calQ$),
  and the operator $\Delta F$
 (derived from $\Gamma$)
 indeed equals the quantum It\^o projection
 $\Delta$ composed with $F$; the constraint~\eqref{structure inequality} on $F$ is then
 equivalent to $F\in C_0(\init,\noise)$, the class defined in~\eqref{Mr bijection}.
 In this sense we have the inclusion
 \[
 \Xfourholhk \supset C_0(\init,\noise).
 \]
 We shall see that this matches up with the inclusion~\eqref{hol contains Mr}.

(iii)
 Clearly the form $\gamma$ may be recovered from $\Gamma$,
 but so may each of the operators $L$, $\Ltilde$ and $C$ --- by polarisation.
 Thus $\Gamma$ determines the quadruple $\mathbb{F}$,
 in particular it determines the operator $\Delta F$.

 (iv)
 In general, we are \emph{not} here viewing the operator $\Delta F$
 as a composition of $\Delta$ with some operator $F$.
\end{remarks}

\begin{propn}
 Let $V\in\QScCholhk$.
 Then the following hold.
 \begin{alist}
 \item
 There are operators $L^V$ and $\Ltilde^V$ from $\init \to
 \init\ot\noise$ with domain $\calQ^V$ such that
 \begin{align*}
 &t^{-1} E^{1_{[0,t[}} V_t E_{\ve(0)} u \to L^Vu
 \ \tu{(}\text{weakly}\tu{)}, \text{ and } \\
 &t^{-1} E^{1_{[0,t[}} V_t^* E_{\ve(0)} u \to \Ltilde^Vu
  \ \tu{(}\text{weakly}\tu{)}, \text{ as }
  t \to 0^+,
 \quad
 u \in \calQ^V.
 \end{align*}

 \item
 The quadruple $\mathbb{F}^V := (\gV, \LV, \LtildeV, \CV - I)$
 belongs to $\Xfourholhk$.
 \item
 $L^{\Vdual} = \Ltilde^V$ so $\Ltilde^{\Vdual} = \LV$ and thus
 \[
 \mathbb{F}^{\Vdual} = \big((\gV)^*, \LtildeV, \LV, (\CV)^* -I\big).
 \]
 \end{alist}
\end{propn}

 This is proved using abstract It\^o integration in Fock space.
 Relative boundedness arguments with Yosida-type approximation,
 at the quadratic form level, now combine with the Markov-regular theory
 and the semigroup characterisation of QS cocycles to yield the central result.

 \begin{thm}
 \label{holomorphic bijection}
 The map
 $V \mapsto \mathbb{F}^V$
is a bijection from $\QScCholhk$ to $\Xfourholhk$.
 \end{thm}

The inverse of this map is naturally denoted
\[
\mathbb{F} \mapsto V^\mathbb{F}
\]
 and referred to as the \emph{quantum stochastic generation map} for
 holomorphic QS cocycles.
 It extends the bijection
 $F \mapsto V^F$ given in~\eqref{Mr bijection},
 by Remark~\ref{cited Remark} (ii).

What is achieved in the above result is the infinitesimal
 characterisation of a large class of QS contraction
 cocycles which includes all those whose expectation semigroup is
 symmetric, as well as the Markov-regular cocycles previously
 characterised. Whereas cocycles in the latter class are all governed
 by QS differential equations, it is not hard to construct examples
 in the former class which are not so-governed.
 Here is the simplest example.

\begin{example}
Let $A$ be a nonnegative, selfadjoint, unbounded operator on $\init$ with dense
range, let $\noise = \Comp$, let $P \in B(\init)$ be the orthogonal
projection with range $\Comp v$ for a vector $v$ in
$\init\setminus\Dom A$, and set $\mathbb{F} = \big( \| A
\cdot\|^2/2, PA, -PA, 0 \big)$.
 Then $\mathbb{F} \in \Xfourholhk$, however the holomorphic QS
 contraction cocycle $V^\mathbb{F}$ is not governed by a QS
 differential equation on any exponential domain. The reason for
 this is that $(PA)^* = AP = 0|_{\{v\}^\perp}$, so that the domain
 on which $V$ satisfies a QS differential equation (\cite{LW camb},
 Theorem 4.2) fails to be dense in $\init$.
\end{example}

We may now put some flesh on Part (a) of Proposition~\ref{3.1*}.

\begin{thm}
Let $V\in \QScCholhk$, and let 
 $\Gamma = \Gamma^V$ be the quadratic form on $\init\ot\khat$ 
 associated with the QS generator $\mathbb{F}^V$ of $V$.
 Then, for each $c,d\in\noise$, the quadratic form generator
 of the $(c,d)$-associated semigroup of $V$ is given by
 \[
\gamma_{c,d}[u] = \Gamma(u\chat,u\dhat) + \chi(c,d)\norm{u}^2,
 \quad
 u\in\calQ^V,
 \]
where $ \chi(c,d) := (\norm{c}^2 + \norm{d}^2)/2 - \ip{c}{d}$.
\end{thm}

% \begin{rem}
 This chimes with Remark~\ref{cited Remark} (ii), 
 % following the definition of $\Xfourholhk$ 
 and also with (4.1) in~\cite{LW camb},
 the corresponding identity for the case where $V$ is governed by 
 a QS diffierential equation with reasonable block matrix operator as coefficient.
%\end{rem}

The next result shows that to each holomorphic QS contraction
cocycle on $\init$ with noise dimension space $\noise$ corresponds a
holomorphic contraction semigroup on $\init \ot \khat$.

\begin{thm}
Let $\mathbb{F} = (\gamma, L, \Ltilde, C-I) \in \Xfourholhk$ and let
$\Gamma$ be the associated quadratic form on $\init\ot\khat$, namely
$\Dom \Gamma = \calQ \op (\init\ot\noise)$, where
 $\calQ := \Dom \gamma$, and
\[
\Gamma[\zeta] := \gamma[u] -
 \big( \ip{\xi}{Lu} + \ip{\Ltilde u}{\xi} + \ip{\xi}{(C-I)\xi}
 \big),
 \quad
\text{ for } \zeta = \binom{u}{\xi} \in \Dom \Gamma.
\]
Then $\Gamma \in \Xholhotkhat$.
\end{thm}

\begin{proof}
Since $\calQ$ is dense in $\init$ and
\[
 \re \Gamma[\zeta] \geq \frac{1}{2} \| \Delta F \zeta \|^2,
 \quad
 \zeta \in \Dom \Gamma = \calQ \op (\init\ot\noise),
\]
 where $\Delta F :=
 \left[\begin{smallmatrix} 0 & 0 \\ L & C-I \end{smallmatrix}\right]$,
 the quadratic form
 $\Gamma$ is densely defined and accretive. The fact that $\Gamma$
 is also closed and semisectorial follows from the easily verified
 equivalence of the norms
 \begin{align*}
 &\zeta \mapsto \big( \re \Gamma[\zeta] + \norm{\zeta}^2 \big)^{1/2},
 \text{ and }
 \\
 &\binom{u}{\xi} \mapsto
 \big( \re \gamma[u] + \norm{u}^2 + \norm{\xi}^2 \big)^{1/2}
 \end{align*}
 on $\Dom \Gamma = \calQ \op (\init\op\noise)$,
 and the fact that the accretive form $\gamma$ is closed and
 semisectorial itself.
\end{proof}

% \begin{rems}
 In view of Remark~\ref{cited Remark} (iii),
 % following the definition of $\QScCholhk$,
  it follows that there is a
bijective correspondence between the collection $\QScCholhk$ and a
class of holomorphic contraction semigroups on $\init\ot\khat$.
The latter semigroups are quite different from the global semigroups
associated with cocycles considered in~\cite{LW cb1}, all of which
are specifically coordinate-dependent and enjoy a Schur-action.
%\end{rems}

We have not so far discussed the question of coisometry and
 unitarity of holomorphic QS contraction cocycles $V$. These
 questions are of interest for their
 (unitality, i.e.\ identity-preserving, and multiplicativity)
 implications for induced cocycles on $B(\init)$, and von
 Neumann subalgebras thereof.
It is not hard to verify that a necessary condition for isometry is
 that equality holds in the inequality governing $\mathbb{F}^V$:
\[
 2 \re \Gamma^V[\zeta] =  \| \Delta F^V \zeta \|^2,
 \quad
 \zeta \in \calQ \op (\init\ot\noise),
\]
 the holomorphic counterpart to the  condition $q(F)=0$ for Markov-regular cocycles.
Whilst this condition is also sufficient in the case of
Markov-regular cocycles, sufficient conditions are trickier in the
holomorphic case. We quote a general result from~\cite{LW cb1} (see
also~\cite{LSi holomorphic}).
 Recall the definition~\eqref{induced cocycle}.

\begin{thm}
 \label{V coisom}
 Let $V\in\QScChk$ and let $\{\calPcd: c,d\in\noise\}$ be the
 associated semigroups of the induced cocycle $k^V$ on $B(\init)$.
 Then, for any total subset $\Til$ of $\noise$ containing $0$, the
 following are equivalent\tu{:}
 \begin{rlist}
 \item
 $V$ is coisometric \tu(equivalently, $k^V$ is unital\tu{)}.
 \item
 $\calPcc$ is unital for all $c\in\Til$.
 \end{rlist}
\end{thm}

% \begin{rem}
 Thus, in
 case $\noise$ is finite dimensional, (ii) need only involve the
verification of unitality of $(1+\dim \noise)$ semigroups.
%\end{rem}

The above result frames a theorem of Fagnola to the effect that, in
case $V$ satisfies a QS differential equation on a core for the
generator of its expectation semigroup, unitality of just the
expectation semigroup $\mathcal{P}^{0,0}$ of $k^V$ suffices
 (\cite{Fagnola}, Theorem 5.23).

We next explain the connection to minimal quantum dynamical
semigroups. By a \emph{quantum dynamical semigroup} on $B(\init)$ is
meant a semigroup $\QDS = (\QDS_t)_{t\geq 0}$ of normal, completely
positive contractions on $B(\init)$ which is continuous in the
point-ultraweak topology.
 Let $\Xtwohk$ denote the collection of pairs $(K,L)$ consisting of
 a $C_0$-semigroup generator $K\in\Xh$ and
 operator $L$ from $\init$ to $\init\ot\noise$ satisfying
 \[
 \Dom L \supset \Dom K \text{ and }
 \norm{Lu}^2 \leq -2 \re\ip{u}{Ku},
 \quad
 u \in \Dom K.
 \]
 For $x \in B(\init)$ define the quadratic form
 $\Lindblad (x) = \Lindblad_{(K,L)}(x)$ on $\init$ by
\begin{align*}
 &\Dom \Lindblad (x) := \Dom K, \\
 &\Lindblad(x)[u] :=
 \ip{Ku}{xu} + \ip{u}{xKu} + \ip{Lu}{(x\ot I_\noise)Lu}.
\end{align*}

A quantum dynamical semigroup $\QDS$ on $B(\init)$ is \emph{minimal}
 for the pair $(K,L)\in\Xtwohk$ if it satisfies:
 \begin{rlist}
 \item
 for all $x\in B(\init)$, $u\in\Dom K$ and $t\in\Rplus$,
 \[
 \ip{u}{\QDS_t(x)u} = \ip{u}{xu} +
 \int_0^t ds\ \Lindblad_{(K,L)}(\QDS_s(x))[u];
 \]
 \item
 if $\QDS'$ is another quantum dynamical semigroup satisfying (i)
 then
 \[
 \QDS'_t(x)) \geq \QDS_t(x) \text{ for all } x \in B(\init)_+, t\in\Rplus.
 \]
 \end{rlist}
Minimal quantum dynamical semigroups are noncommutative counterparts
of Feller's minimal solutions of the Fokker-Planck equation.

\begin{thm}[\cite{Davies minimal}]
 Let $(K,L)\in\Xtwohk$.
 Then there is a unique minimal quantum dynamical semigroup
 $\minimalKL$ associated with $(K,L)$.
\end{thm}

% \begin{rem}
  Note that,
 if $\minimalKL$ is \emph{conservative}, that is unital:
 $\minimalKL_t(I)=I$ ($t\in\Rplus$), then
 $\Lindblad_{(K,L)}(I) = 0$:
 \[
 \norm{Lu}^2 + 2 \re \ip{u}{Ku} = 0,
 \quad
 u \in \Dom K.
 \]
% \end{rem}

In the spirit of the inclusion $\Xholh \subset \Xh$ we may consider
the subclass $\Xtwoholhk$ of $\Xtwohk$ consisting of pairs
$(\gamma,L)$ where $\gamma\in\Xholh$ and $(K,L) \in \Xtwohk$,
$-K$ being the operator corresponding to $\gamma$.
 In this case we write $\QDS^{(\gamma,L)}$ for the
 associated minimal semigroup on $B(\init)$.
In view of the relation
\[
 \norm{Lu}^2 \leq 2 \re \gamma[u], \quad u \in \Dom \gamma,
\]
 for $\mathbb{F} = (\gamma, L, \Ltilde, C-I) \in \Xfourholhk$, any such $\mathbb{F}$ 
truncates to a pair $(\gamma,L) \in \Xtwoholhk$.
 Conversely, any pair $(\gamma,L) \in \Xtwoholhk$ `dilates' to a
 quadruple
 \[
 \mathbb{F} = (\gamma, L, -C^*L, C-I)\in\Xfourholhk
\]
 by
 choosing a unitary, or contraction, $C$ in $B(\init;\noise)$;
 the choices $C=I$, respectively
 $C=0$ or $C=-I$, being notable ones.
 The next result underpins these considerations.

 \begin{thm}
 \label{underpins}
 Let $V = V^\mathbb{F}$ for
 $\mathbb{F} = (\gamma, L, \Ltilde, C-I) \in \Xfourholhk$.
 Then the expectation semigroup of the induced cocycle on $B(\init)$
 of the dual semigroup $\Vdual = V^{\wt{\mathbb{F}}}$ coincides with
 the minimal semigroup associated with the truncation
 $(\gamma, L) \in \Xtwoholhk$ of $\mathbb{F}$,
 equivalently,
 \[
 \mathbb{E}\big[ V^*_t (x\ot I_\Fock) V_t \big] =
 \QDS^{(\gamma,L)}(x),
 \quad
 x\in B(\init), t\in\Rplus.
 \]
 \end{thm}

% \begin{rem}
This is the holomorphic counterpart to a corresponding result for QS
 contraction cocycles governed by a QS differential equation
(\cite{Fagnola}, Theorem 5.22).
 The relationship between Fagnola's analysis --- which uses finite
 particle vectors (rather than exponential vectors) in a crucial
 way,
 and builds on earlier work of Mohari, Sinha and others --- and our
 analysis, remains an intriguing one.
% \end{rem}

 To end I shall briefly describe
 a new tool for the construction and
 analysis of QS cocycles --- the \emph{QS Trotter product formula}
 (\cite{LSi Trotter1}).
 In brief what this achieves is the representation of the QS cocycle
 $V^F$, where $F$ has the form
% with stochastic generator of the form
 \[
 F = F_1 \boxplus F_2 :=
 \begin{bmatrix}
 K_1 + K_2 & M_1 & M_2 \\ L_1 & C_1 - I_1 & \\ L_2 && C_2 - I_2
 \end{bmatrix}
 % \in C_0(\init,\noise),
  \in B(\init\ot\khat) = B(\init \op \init\ot\noise_1 \op \init \ot \noise_2)
 \]
  with $\noise = \noise_1\op\noise_2$ and
 $I_i$ denoting the identity operator on $\init\ot\noise_i$ ($i=1,2$),
 in terms of the cocycles $V^{F_1}$ and $V^{F_2}$, where
 \[
 F_i =
 \begin{bmatrix}
 K_i & M_i \\ L_i & C_i - I_i
 \end{bmatrix}
 \in
   B(\init\ot\khat_i) = B(\init \op \init\ot\noise_i)
% C_0(\init,\noise_i),
 \text{ for } i=1,2.
\]
 It is easily verified that, if $F_1\in  C_0(\init,\noise_1)$ and
  $F_2\in  C_0(\init,\noise_2)$, then  $F\in  C_0(\init,\noise)$.

% \begin{rem}
 In the quantum control literature the composition $\boxplus$,
 known as the \emph{concatenation product},
 has application to quantum 
 feedback % a final (post Kuo) edit insert
 networks,
 particularly in combination with the \emph{series product}
 (\cite{GoJ}). The latter corresponds to the
 Evans-Hudson perturbation formula
 (\cite{EvH}, \cite{DaS}, \cite{GLW}, \cite{BLS}) specialised
 to the case where the `free' QS flow is implemented by a
 (Markov-regular) QS  unitary cocycle $V$;
 the perturbed QS flow is implemented by
 the unitary cocycle whose generator is
 (up to the choice of parameterisation)
 the series product
 of the stochastic generator of $V$ and the perturbation coefficient.
 % \end{rem}

 The QS Trotter product formula
  extends to both mapping cocycles and to holomorphic
 contraction operator cocycles  (\cite{LSi Trotter2}).
In the latter case this is given by
 \[
 \mathbb{F} =  \mathbb{F}_1 \boxplus  \mathbb{F}_2 =
  (\gamma, L, \Ltilde, C-I) \in \Xfourholhk
\]
 where, for
 $(\gamma_i, L_i, \Ltilde_i, C_i-I_i) \in \Xfourholhki$ ($i=1,2$),
\[
 \gamma = \gamma_1 + \gamma_2, \
 L = \begin{bmatrix}L_1 \\ L_2 \end{bmatrix}, \
 \Ltilde = \begin{bmatrix}\Ltilde_1 \\ \Ltilde_2 \end{bmatrix}, \
 C-I = \begin{bmatrix}C_1-I_1 &\\& C_2-I_2 \end{bmatrix},
\]
 yielding $V^{\mathbb{F}}$ in
 terms of $V^{\mathbb{F}_1}$ and $V^{\mathbb{F}_2}$.
 The only constraint for forming this concatenation product is that the intersection
 $\Dom \gamma_1 \cap \Dom \gamma_2$ is dense in $\init$ since
 the sum of two closed, accretive, semisectorial forms is
 closed, accretive and semisectorial, and we have the following identities,
 by Pythagoras:
 \[
 \Gamma[\zeta] = \Gamma_1[\eta_1] +  \Gamma_2[\eta_2] \ \text{ and }\
  \norm{\Delta F\zeta} =
  \norm{\Delta F_1\eta_1} +  \norm{\Delta F_2\eta_2}
\]
 where, for $\zeta = \binom{u}{\xi}$ and $\xi=\binom{\xi_1}{\xi_2}$,
 $\eta_1 := \binom{u}{\xi_1}$ and $\eta_2 := \binom{u}{\xi_2}$.
 In the case of mapping cocycles, the
  homomorphic property of cocycles so-constructed on an
 operator algebra is investigated in~\cite{DGS}.

 In all of this work,  the key ingredients are
 Trotter products of the associated semigroups
 of the constituent cocycles.

\bigskip
\noindent
\emph{ACKNOWLEDGEMENT.} The joint work with Kalyan Sinha described
in Section~\ref{section: holomorphic} is supported by the UKIERI
Research Collaboration Network grant \emph{Quantum Probability,
Noncommutative Geometry \& Quantum Information}.

%\begin{thebibliography}{BF+} %%%% failed attempt to deal with Latex error!\widehat{}

% $ deliberate mistake

\end{document}